\documentclass[11pt,english]{article}

\usepackage[T1]{fontenc}
\usepackage[latin9]{inputenc}
\usepackage{geometry}
\usepackage{babel}
\usepackage{amsmath}
\usepackage{amsthm}
\usepackage{amssymb}
\usepackage{setspace}
\usepackage[unicode=true,pdfusetitle,
 bookmarks=true,bookmarksnumbered=false,bookmarksopen=false,
 breaklinks=false,pdfborder={0 0 0},backref=false,colorlinks=false]
 {hyperref}

\makeatletter
\theoremstyle{plain}
\newtheorem{thm}{\protect\theoremname}
  \theoremstyle{plain}
  \newtheorem{cor}[thm]{\protect\corollaryname}
  \theoremstyle{plain}
  \newtheorem{lem}[thm]{\protect\lemmaname}
  \theoremstyle{plain}
  \newtheorem{prop}[thm]{\protect\propositionname}
  \theoremstyle{plain}
  \newtheorem{conjecture}[thm]{\protect\conjecturename}
  \theoremstyle{definition}
  \newtheorem{defn}[thm]{\protect\definitionname}
  \theoremstyle{definition}
  \newtheorem{example}[thm]{\protect\examplename}
  \theoremstyle{remark}
  \newtheorem{rem}[thm]{\protect\remarkname}
  \theoremstyle{remark}
  \newtheorem{claim}[thm]{\protect\claimname}

\date{}

\usepackage[nameinlink,capitalise,noabbrev]{cleveref}
\AtBeginDocument{\renewcommand{\ref}[1]{\cref{#1}}}

\usepackage[parfill]{parskip}
\begingroup
    \makeatletter
    \@for\theoremstyle:=definition,remark,plain\do{%
        \expandafter\g@addto@macro\csname th@\theoremstyle\endcsname{%
            \addtolength\thm@preskip\parskip
            }%
        }
\endgroup

\theoremstyle{plain}
\newtheorem{mythm}{\protect\theoremname}
\renewenvironment{thm}{\begin{mythm}}{\end{mythm}}
\crefname{mythm}{Theorem}{Theorems}
\theoremstyle{definition}
\newtheorem{mydefn}[mythm]{\protect\definitionname}
\renewenvironment{defn}{\begin{mydefn}}{\end{mydefn}}
\theoremstyle{definition}
\newtheorem{myexample}[mythm]{\protect\examplename}
\renewenvironment{example}{\begin{myexample}}{\end{myexample}}
\theoremstyle{plain}
\newtheorem{myprop}[mythm]{\protect\propositionname}
\renewenvironment{prop}{\begin{myprop}}{\end{myprop}}
\theoremstyle{plain}
\newtheorem{mycor}[mythm]{\protect\corollaryname}
\renewenvironment{cor}{\begin{mycor}}{\end{mycor}}
\theoremstyle{plain}
\newtheorem{mylem}[mythm]{\protect\lemmaname}
\renewenvironment{lem}{\begin{mylem}}{\end{mylem}}
\crefname{mylem}{Lemma}{Lemmas}
\theoremstyle{plain}
\newtheorem{myconjecture}[mythm]{\protect\conjecturename}
\renewenvironment{conjecture}{\begin{myconjecture}}{\end{myconjecture}}
\theoremstyle{remark}
\newtheorem{myrem}[mythm]{\protect\remarkname}
\renewenvironment{rem}{\begin{myrem}}{\end{myrem}}
\theoremstyle{plain}
\newtheorem{myclaim}[mythm]{\protect\claimname}
\renewenvironment{claim}{\begin{myclaim}}{\end{myclaim}}
\theoremstyle{definition}

\crefformat{equation}{#2(#1)#3}

\let\originalleft\left
\let\originalright\right
\renewcommand{\left}{\mathopen{}\mathclose\bgroup\originalleft}
\renewcommand{\right}{\aftergroup\egroup\originalright}
\usepackage{pgfplots}
\usetikzlibrary{pgfplots.groupplots}
\usepackage{verbatim}

\makeatletter
\renewcommand*{\UrlTildeSpecial}{%
  \do\~{%
    \mbox{%
      \fontfamily{ptm}\selectfont
      \textasciitilde
    }%
  }%
}%
\let\Url@force@Tilde\UrlTildeSpecial
\makeatother

\usepackage{tikz}
\usetikzlibrary{decorations.markings}
\tikzstyle{vertex}=[circle,draw=black,fill=black,inner sep=0,minimum size=0.2cm,text=white,font=\footnotesize]

\usepackage[labelfont=bf,labelsep=period]{caption}

\usepackage[nomessages]{fp}

\usepackage{calc}

\makeatother

  \providecommand{\claimname}{Claim}
  \providecommand{\conjecturename}{Conjecture}
  \providecommand{\corollaryname}{Corollary}
  \providecommand{\definitionname}{Definition}
  \providecommand{\examplename}{Example}
  \providecommand{\lemmaname}{Lemma}
  \providecommand{\propositionname}{Proposition}
  \providecommand{\remarkname}{Remark}
\providecommand{\theoremname}{Theorem}

\begin{document}

\title{Bounded-degree spanning trees in randomly perturbed graphs}

\author{Michael Krivelevich \thanks{School of Mathematical Sciences, Raymond and Beverly Sackler Faculty
of Exact Sciences, Tel Aviv University, 6997801, Israel. Email: \protect\href{mailto:krivelev@post.tau.ac.il}{krivelev@post.tau.ac.il}.
Research supported in part by USA-Israel BSF Grant 2010115 and by
grant 912/12 from the Israel Science Foundation.}\and Matthew Kwan\thanks{Department of Mathematics, ETH, 8092 Zurich. Email: \protect\href{mailto:matthew.kwan@math.ethz.ch}{matthew.kwan@math.ethz.ch}}\and
Benny Sudakov\thanks{Department of Mathematics, ETH, 8092 Zurich. Email: \protect\href{mailto:benjamin.sudakov@math.ethz.ch}{benjamin.sudakov@math.ethz.ch}.
Research supported in part by SNSF grant 200021-149111.}}

\maketitle
\global\long\def\RR{\mathbb{R}}

\global\long\def\QQ{\mathbb{Q}}

\global\long\def\HH{\mathbb{H}}

\global\long\def\E{\mathbb{E}}

\global\long\def\Var{\operatorname{Var}}

\global\long\def\CC{\mathbb{C}}

\global\long\def\NN{\mathbb{N}}

\global\long\def\ZZ{\mathbb{Z}}

\global\long\def\GG{\mathbb{G}}

\global\long\def\BB{\mathbb{B}}

\global\long\def\DD{\mathbb{D}}

\global\long\def\cL{\mathcal{L}}

\global\long\def\supp{\operatorname{supp}}

\global\long\def\one{\boldsymbol{1}}

\global\long\def\range#1{\left[#1\right]}

\global\long\def\d{\operatorname{d}}

\global\long\def\falling#1#2{\left(#1\right)_{#2}}

\global\long\def\f{\mathbf{f}}

\global\long\def\im{\operatorname{im}}

\global\long\def\sp{\operatorname{span}}

\global\long\def\sign{\operatorname{sign}}

\global\long\def\mod{\operatorname{mod}}

\global\long\def\id{\operatorname{id}}

\global\long\def\disc{\operatorname{disc}}

\global\long\def\lindisc{\operatorname{lindisc}}

\global\long\def\tr{\operatorname{tr}}

\global\long\def\adj{\operatorname{adj}}

\global\long\def\Unif{\operatorname{Unif}}

\global\long\def\Po{\operatorname{Po}}

\global\long\def\Bin{\operatorname{Bin}}

\global\long\def\Ber{\operatorname{Ber}}

\global\long\def\Geom{\operatorname{Geom}}

\global\long\def\Hom{\operatorname{Hom}}

\global\long\def\floor#1{\left\lfloor #1\right\rfloor }

\global\long\def\ceil#1{\left\lceil #1\right\rceil }

\global\long\def\flip#1{\overline{#1}}

\global\long\def\a{\alpha}

\global\long\def\c{c}

\global\long\def\D{\Delta}

\global\long\def\G{G}

\global\long\def\F{F}

\global\long\def\T{T}

\global\long\def\R{R}

\global\long\def\B#1{\Gamma_{A,B}\left(#1\right)}

\begin{abstract}
We show that for any fixed dense graph $G$ and bounded-degree tree
$T$ on the same number of vertices, a modest random perturbation
of $G$ will typically contain a copy of $T$. This combines the viewpoints
of the well-studied problems of embedding trees into fixed dense graphs
and into random graphs, and extends a sizeable body of existing research
on randomly perturbed graphs. Specifically, we show that there is
$\c=\c\left(\a,\D\right)$ such that if $\G$ is an $n$-vertex graph
with minimum degree at least $\a n$, and $\T$ is an $n$-vertex
tree with maximum degree at most $\D$, then if we add $\c n$ uniformly
random edges to $\G$, the resulting graph will contain $\T$ asymptotically
almost surely (as $n\to\infty$). Our proof uses a lemma concerning
the decomposition of a dense graph into super-regular pairs of comparable
sizes, which may be of independent interest.
\end{abstract}

\section{Introduction}

A classical theorem of Dirac \cite{Dir52} states that any $n$-vertex
graph ($n\ge3$) with minimum degree at least $n/2$ has a Hamilton
cycle: a cycle that passes through all the vertices of the graph.
More recently, there have been many results showing that this kind
of minimum degree condition is sufficient to guarantee the existence
of different kinds of spanning subgraphs. For example, in \cite[Theorem~1$'$]{KSS95}
Koml\'os, S\'ark\"ozy and Szemer\'edi proved that for any $\Delta$
and $\gamma>0$, any $n$-vertex graph ($n$ sufficiently large) with
minimum degree $\left(1/2+\gamma\right)n$ contains every spanning
tree which has maximum degree at most $\Delta$. This has also been
generalized further in two directions. In \cite{KSS01} Koml\'os,
S\'ark\"ozy and Szemer\'edi improved their result to allow $\Delta$
to grow with $n$, and in \cite{BST09} B\"ottcher, Schacht and Taraz
gave a result for much more general spanning subgraphs than trees
(see also a result \cite{Csa08} due to Csaba).

The constant $1/2$ in the above Dirac-type theorems is tight: in
order to guarantee the existence of these spanning subgraphs we require
very strong density conditions. But the situation is very different
for a ``typical'' large graph. If we fix an arbitrarily small $\alpha>0$
and select a graph uniformly at random among the (labelled) graphs
with $n$ vertices and $\alpha{n \choose 2}$ edges, then the degrees
will probably each be about $\alpha n$. Such a random graph is Hamiltonian
with probability $1-o\left(1\right)$ (we say it is Hamiltonian \emph{asymptotically
almost surely}, or \emph{a.a.s.}). This follows from a stronger result
by P\'osa \cite{Pos76} that gives a \emph{threshold} for Hamiltonicity:
a random $n$-vertex, $m$-edge graph is Hamiltonian a.a.s. if $m\gg n\log n$,
and fails to be Hamiltonian a.a.s. if $m\ll n\log n$. Although the
exact threshold for bounded-degree spanning trees is not known, in
\cite{Mon14} Montgomery proved that for any $\Delta$ and any tree
$T$ with maximum degree at most $\Delta$, a random $n$-vertex graph
with $\Delta n\left(\log n\right)^{5}$ edges a.a.s. contains $T$.
Here and from now on, all asymptotics are as $n\to\infty$, and we
implicitly round large quantities to integers.

In \cite{BFM03}, Bohman, Frieze and Martin studied Hamiltonicity
in the random graph model that starts with a dense graph and adds
$m$ random edges. This model is a natural generalization of the ordinary
random graph model where we start with nothing, and offers a ``hybrid''
perspective combining the extremal and probabilistic settings of the
last two paragraphs. The model has since been studied in a number
of other contexts; see for example \cite{BHM04,KST06,KKS15}. A property
that holds a.a.s. with small $m$ in our random graph model can be
said to hold not just for a ``globally'' typical graph but for the
typical graph in a small ``neighbourhood'' of our space of graphs.
This tells us that the graphs which fail to satisfy our property are
in some sense ``fragile''. We highlight that it is generally very
easy to transfer results from our model of random perturbation to
other natural models, including those that delete as well as add edges
(this can be accomplished with standard coupling and conditioning
arguments). We also note that a particularly important motivation
is the notion of \emph{smoothed analysis} of algorithms introduced
by Spielman and Teng in \cite{ST04}. This is a hybrid of worst-case
and average-case analysis, studying the performance of algorithms
in the more realistic setting of inputs that are ``noisy'' but not
completely random.

The statement of \cite[Theorem~1]{BFM03} is that for every $\alpha>0$
there is $c=c\left(\alpha\right)$ such that if we start with a graph
with minimum degree at least $\alpha n$ and add $cn$ random edges,
then the resulting graph will a.a.s. be Hamiltonian. This saves a
logarithmic factor over the usual model where we start with the empty
graph. Note that some dense graphs require a linear number of extra
edges to become Hamiltonian (consider for example the complete bipartite
graph with partition sizes $n/3$ and $2n/3$), so the order of magnitude
of this result is tight.

Let $\GG\left(n,p\right)$ be the binomial random graph model with
vertex set $\range n=\left\{ 1,\dots,n\right\} $, where each edge
is present independently and with probability $p$. (For our purposes
this is equivalent to the model that uniformly at random selects a
$p{n \choose 2}$-edge graph on the vertex set $\range n$). In this
paper we prove the following theorem, extending the aforementioned
result to bounded-degree spanning trees.
\begin{thm}
\label{thm:main-theorem}There is $\c=\c\left(\a,\D\right)$ such
that if $\G$ is a graph on the vertex set $\range n$ with minimum
degree at least $\a n$, and $\T$ is an $n$-vertex tree with maximum
degree at most $\D$, and $\R\in\GG\left(n,\c/n\right)$, then a.a.s.
$\T\subseteq\G\cup\R$.
\end{thm}
One of the ingredients of the proof is the following lemma, due to
Alon and two of the authors (\cite[Theorem~1.1]{AKS07}). It shows
that the random edges in $\R$ are already enough to embed almost-spanning
trees.
\begin{lem}
\label{lem:almost-spanning}There is $\c=\c\left(\varepsilon,\D\right)$
such that $\G\in\GG\left(n,\c/n\right)$ a.a.s. contains every tree
of maximum degree at most $\D$ on $\left(1-\varepsilon\right)n$
vertices.
\end{lem}
We will split the proof of \ref{thm:main-theorem} into two cases
in a similar way to \cite{Kri10}. If our spanning tree $\T$ has
many leaves, then we remove the leaves and embed the resulting non-spanning
tree in $\R$ using \ref{lem:almost-spanning}. To complete this into
our spanning tree $\T$, it remains to match the vertices needing
leaves with leftover vertices. We show that this is possible by verifying
that a certain Hall-type condition typically holds.

The more difficult case is where $\T$ has few leaves, which we attack
in \ref{sub:few leaves}. In this case $\T$ cannot be very ``complicated''
and must be a subdivision of a small tree. In particular $\T$ must
have many long \emph{bare paths}: paths where each vertex has degree
exactly two. By removing these bare paths we obtain a small forest
which we can embed into $\R$ using \ref{lem:almost-spanning}. In
order to complete this forest into our spanning tree $\T$, we need
to join up distinguished ``special pairs'' of vertices with disjoint
paths of certain lengths.

In order to make this task feasible, we first use Szemer\'edi's regularity
lemma to divide the vertex set into a bounded number of ``super-regular
pairs'' (basically, dense subgraphs with edges very well-distributed
in $\G$). After embedding our small forest and making a number of
further adjustments, we can find our desired special paths using super-regularity
and a tool called the ``blow-up lemma''.

\section{Proof of \texorpdfstring{}{Theorem~}\ref{thm:main-theorem}}

We split the random edges into multiple independent ``phases'':
say $\R\supseteq\R_{1}\cup\R_{2}\cup\R_{3}\cup\R_{4}$, where $\R_{i}\in\GG\left(n,\c_{i}/n\right)$
for some large $\c_{i}=\c_{i}\left(\a,\D\right)$ to be determined
later (these $\c_{i}$ will in turn determine $\c$). Let $V=\range n$
be the common vertex set of $\G$, $\R$ and the $\R_{i}$.
\begin{rem}
\label{rem:symmetry}At several points in the proof we will prove
that (with high probability) there exist certain substructures in
the random edges $\R$, and we will assume by symmetry that these
substructures (or some corresponding vertex sets) are themselves uniformly
random. A simple way to see why this kind of reasoning is valid is
to notice that if we apply a uniformly random vertex permutation to
a random graph $\R\in\GG\left(n,p\right)$, then by symmetry the resulting
distribution is still $\GG\left(n,p\right)$. So we can imagine that
we are finding substructures in an auxiliary random graph, then randomly
mapping these structures to our vertex set.
\end{rem}

\subsection{Case 1: $\protect\T$ has many leaves\label{sub:many-leaves}}

For our first case suppose there are at least $\lambda n$ leaves
in $\T$ (for some fixed $\lambda=\lambda\left(\a\right)>0$ which
will be determined in the second case). Then consider a tree $\T'$
with some $\lambda n$ leaves removed. By \ref{lem:almost-spanning}
we can a.a.s. embed $\T'$ into $\R_{1}$ (abusing notation, we also
denote this embedded subgraph by $\T'$). Let $A\subseteq V$ be the
set of vertices of $\T'$ which had a leaf deleted from them, and
for $a\in A$ let $\ell\left(a\right)\le\Delta$ be the number of
leaves that were deleted from $a$. Let $B\subseteq V$ be the set
of $\lambda n$ vertices not part of $\T'$. In order to complete
the embedding of $\T'$ into an embedding of $\T$ it suffices to
find a set of disjoint stars in $\G\cup\R_{2}$ each with a center
vertex $a\in A$ and with $\ell\left(a\right)$ leaves in $B$. We
will prove the existence of such stars with the following Hall-type
condition.
\begin{lem}
\label{lem:hall-stars}Consider a bipartite graph $\G$ with bipartition
$A\cup B$, and consider a function $\ell:A\to\NN$ such that $\sum_{a\in A}\ell\left(a\right)=\left|B\right|$.
Suppose that 
\[
\left|N\left(S\right)\right|\ge\sum_{s\in S}\ell\left(s\right)\quad\mbox{for all }S\subseteq A.
\]
Then there is a set of disjoint stars in $\G$, each of which has
a center vertex $a\in A$ and $\ell\left(a\right)$ leaves in $B$.
\end{lem}
\ref{lem:hall-stars} can be easily proved by applying Hall's marriage
theorem to an auxiliary bipartite graph which has $\ell\left(a\right)$
copies of each vertex $a\in A$.

In this section, and at several points later in the paper, we will need to consider the intersection of random sets with fixed sets. The following concentration inequality (taken from \cite[Theorem 2.10]{JLR00}) will be useful.

\begin{lem}\label{lem:hypergeometric-concentration}For a set $\Gamma$ with $|\Gamma|=N$, a fixed subset $\Gamma'\subseteq \Gamma$ with $\Gamma'=m$, and a uniformly random $n$-element subset $\Gamma_n\subseteq \Gamma$, define the random variable $X=|\Gamma_n\cap\Gamma'|$. This random variable is said to have the \emph{hypergeometric distribution} with parameters $N$,$n$ and $m$. Noting that $\E X=mn/N$, we have
\[
\Pr\left(|X-\E X|\ge \varepsilon \E X\right)\le \exp(c_\varepsilon \E X)
\]
for some constant $c_\varepsilon$ depending on $\varepsilon$.
\end{lem}

Now, in accordance with \ref{rem:symmetry} we can assume $A$ and $B$
are uniformly random disjoint subsets of $V$ of their size (note
$\left|A\right|\ge\lambda n/\D$ and $\left|B\right|=\lambda n$),
not depending on $\G$. For each $a\in A$, the random variable $\left|N_{\G}\left(a\right)\cap B\right|$ is hypergeometrically distributed with expected value at least $\lambda\a n$,
and similarly each $\left|N_{\G}\left(b\right)\cap A\right|$ ($b\in B$)
is hypergeometric with expectation at least $\lambda\a n/\D$. Let
$\beta=\lambda\a/\left(2\D\right)$; by \ref{lem:hypergeometric-concentration}
and the union bound, in $\G$ each $a\in A$ a.a.s. has at least $\beta n$
neighbours in $B$, and each $b\in B$ a.a.s. has at least $\beta n$
neighbours in $A$. That is to say, the bipartite graph induced by
$\G$ on the bipartition $A\cup B$ has minimum degree at least $\beta n$.
We treat $A$ and $B$ as fixed sets satisfying this property. It
suffices to prove the following lemma.
\begin{lem}
\label{lem:bipartite-perturbed-stars}For any $\beta>0$ and $\Delta\in\RR$
there is $\c=\c\left(\beta\right)$ such that the following holds.
Suppose $\G$ is a bipartite graph with bipartition $A\cup B$ ($\left|A\right|,\left|B\right|\le n$)
and minimum degree at least $\beta n$. Suppose $\ell:A\to\range{\Delta}$
is some function satisfying $\sum_{a\in A}\ell\left(a\right)=\left|B\right|$.
Let $\R$ be the random bipartite graph where each edge between $A$
and $B$ is present independently and with probability $\c/n$. Then
a.a.s. $\G\cup\R$ satisfies the condition in \ref{lem:hall-stars}.\end{lem}
\begin{proof}
Let $S\subseteq A$ be nonempty. If $\left|S\right|\le\beta n/\Delta$
then $\left|N_{\G}\left(S\right)\right|\ge\left|N_{\G}\left(a\right)\right|\ge\beta n\ge\sum_{s\in S}\ell\left(s\right)$
for any $a\in S$. If $\left|S\right|\ge\left|A\right|-\beta n$ then
$\left|A\backslash S\right|\le\left|N_{\G}\left(b\right)\right|$
for each $b\in B$, so $\left|N_{\G}\left(S\right)\right|=\left|B\right|\ge\sum_{s\in S}\ell\left(s\right)$.

The remaining case is where $\beta n/\Delta\le\left|S\right|\le\left|A\right|-\beta n$.
In this case $\sum_{s\in S}\ell\left(s\right)=\left|B\right|-\sum_{a\in A\backslash S}\ell\left(a\right)\le\left|B\right|-\beta n$.
Now, for any subsets $A'\subseteq A$ and $B'\subseteq B$ with $\left|A'\right|=\beta n/\Delta$
and $\left|B'\right|=\beta n$, the probability there is no edge between
$A'$ and $B'$ is $\left(1-\c/n\right)^{\left(\beta n/\Delta\right)\left(\beta n\right)}\le e^{-\c\beta^{2}n/\Delta}$.
There are at most $2^{2n}$ choices of such $A',B'$, so for large
$\c$, by the union bound there is a.a.s. an edge between any such
pair of sets. If this is true, it follows that $\left|N_{\R}\left(S\right)\right|\ge\left|B\right|-\beta n$,
because if $S$ had more than $\beta n$ non-neighbours in $B$ then
this would give us a contradictory pair of subsets with no edge between
them. We have proved that a.a.s. $\left|N_{\R}\left(S\right)\right|\ge\sum_{s\in S}\ell\left(s\right)$
for all $S\subseteq A$, as required.
\end{proof}

\subsection{Case 2: $\protect\T$ has few leaves\label{sub:few leaves}}

\global\long\def\h{h}

Now we address the second case where there are fewer than $\lambda n$
leaves in $\T$. The argument is broken into a number of subsections,
which we outline here. The first step (carried out in \ref{sub:partition-into-blobs})
is to divide $\G$ into a bounded number of pairs of ``partner clusters''
with edges well-distributed between them. Specifically, the edges
between each pair of partner clusters will each have a property called
``super-regularity''\emph{.} The significance of this is that one
can use a tool called the ``blow up lemma'' to easily embed bounded-degree
spanning structures into super-regular pairs.

Guided by our partition into pairs of clusters, we then embed most
of $\T$ into $\G\cup\R$ in a step-by-step fashion. The objective
is to do this in such a way that, afterwards, the remaining connected
components of $\T$ can each be assigned to their own cluster-pair,
so that we can finish the embedding in each cluster-pair individually,
with the blow-up lemma.

As outlined in the introduction, the fact that $\T$ has few leaves
means that it is mostly comprised of long \emph{bare paths}: paths
where every vertex has degree two. The first step in the embedding
is to embed the non-bare-path parts of $\T$ using only the random
edges in $\R$. That is, we obtain a forest $\F$ by removing bare
paths from $\T$, then we embed $\F$ into $\R$ using \ref{lem:almost-spanning}.
This step is performed in \ref{sub:embed-forest}.

After embedding $\F$, it remains to connect certain ``special pairs''
of vertices with paths of certain lengths in $\G\cup\R$, using the
vertices that were not yet used to embed $\F$. (That is, we need
to embed the bare paths we deleted from $\T$ to obtain $\F$). There
are a few problems we need to overcome before we can apply the blow-up
lemma to accomplish this. First, the vertices in special pairs, in
general, lie in totally different clusters in our partition, so we
cannot hope to embed each path in a single cluster-pair. In \ref{sub:fix-endpoints},
we correct this problem by finding very short paths from the vertices
of each special pair to a common cluster-pair.

The next issue is that the relative sizes of the cluster-pairs will
not in general match the number of special pairs they contain. For
example, if a cluster-pair contains $N$ special pairs which we need
to connect with paths of length $k$, then we need that cluster-pair
to have exactly $\left(k-1\right)N$ vertices that were not used for
anything so far. To fix this problem, in \ref{sub:cluster-adjustment}
we already start to embed some of the paths between special pairs,
mainly using the random edges in $\R$. We choose the vertices for
these paths in a very specific way, to control the relative quantities
of remaining vertices in the clusters.

After these adjustments, in \ref{sub:blow-up} we are able to complete
the embedding of $\T$. We show that finding paths between distinguished
vertices is equivalent to finding cycles with certain properties in
a certain auxiliary graph. These cycles can be found in a straightforward
manner using the blow-up lemma and super-regularity.

\subsubsection{Partitioning into super-regular pairs\label{sub:partition-into-blobs}}

As outlined, we first need to divide $\G$ into a bounded number of
pairs of ``partner clusters'' with well-distributed edges.
\begin{defn}
For a disjoint pair of vertex sets $\left(X,Y\right)$ in a graph
$G$, let its \emph{density} $d\left(X,Y\right)$ be the number of
edges between $X$ and $Y$, divided by $\left|X\right|\left|Y\right|$.
A pair of vertex sets $\left(V_{1},V_{2}\right)$ is said to be \emph{$\varepsilon$-regular}
in $G$ if for any $U_{1},U_{2}$ with $U_{\h}\subseteq V_{\h}$ and
$\left|U_{\h}\right|\ge\varepsilon\left|V_{\h}\right|$, we have $\left|d\left(U_{1},U_{2}\right)-d\left(V_{1},V_{2}\right)\right|\le\varepsilon$.
If alternatively $d\left(U_{1},U_{2}\right)\ge\delta$ for all such
pairs $U_{1},U_{2}$ then we say $\left(V_{1},V_{2}\right)$ is \emph{$\left(\varepsilon,\delta\right)$-dense}.
Let $\flip{\h}=2-\h$; if $\left(V_{1},V_{2}\right)$ is $\left(\varepsilon,\delta\right)$-dense
and moreover each $v\in V_{\h}$ has at least $\delta\left|V_{\flip{\h}}\right|$
neighbours in $V_{\flip{\h}}$, then we say $\left(V_{1},V_{2}\right)$
is \emph{$\left(\varepsilon,\delta\right)$-super-regular}.\end{defn}
\begin{lem}
\label{lem:partition-into-blobs}For $\a,\varepsilon>0$ with $\varepsilon$
sufficiently small relative to $\a$, there are $\delta=\delta\left(\a\right)>0$,
$\rho=\rho\left(\a\right)$ and $Q=Q\left(\a,\varepsilon\right)$,
such that the following holds. Let $\G$ be an $n$-vertex graph with
minimum degree at least $\a n$. Then there is $q\le Q$ and a partition
of $V\left(\G\right)$ into clusters $V_{i}^{\h}$ ($1\le i\le q$,
$\h=1,2$) such that each pair $\left(V_{i}^{1},V_{i}^{2}\right)$
is $\left(\varepsilon,\delta\right)$-super-regular. Moreover, each
$\left|V_{i}^{\h}\right|/\left|V_{j}^{g}\right|\le\rho$.
\end{lem}
We emphasize that in \ref{lem:partition-into-blobs} we do \emph{not
}guarantee that the sets $V_{i}^{\h}$ are of the same size, just
that the variation in the sizes of the clusters is bounded independently
of $\varepsilon$.

To prove \ref{lem:partition-into-blobs}, we will apply Szemer\'edi's
regularity lemma to obtain a reduced \emph{cluster graph}, then decompose
this cluster graph into small stars. In each star $T_{i}$, the center
cluster will give us $V_{i}^{1}$ and the leaf clusters will be combined
to form $V_{i}^{2}$. We will then have to redistribute some of the
vertices between the clusters to ensure super-regularity. Before giving
the details of the proof, we give a statement of (a version of) Szemer\'edi's
regularity lemma and some auxiliary lemmas for working with regularity
and super-regularity.
\begin{lem}[Szemer\'edi's regularity lemma, minimum degree form]
\label{lem:szemeredi-regularity}For every $\a>0$, and any $\varepsilon>0$
that is sufficiently small relative to $\a$, there are $\a'=\a'\left(\a\right)>0$
and $K=K\left(\varepsilon\right)$ such that the following holds.
For any graph $\G$ of minimum degree at least $\a\left|\G\right|$,
there is a partition of $V\left(\G\right)$ into clusters $V_{0},V_{1},\dots V_{k}$
($k\le K$), and a spanning subgraph $\G'$ of $\G$, satisfying the
following properties. The ``exceptional cluster'' $V_{0}$ has size
at most $\varepsilon n$, and the other clusters have equal size $sn$.
The minimum degree of $\G'$ is at least $\a'n$. There are no edges
of $\G'$ within the clusters, and each pair of non-exceptional clusters
is $\varepsilon$-regular in $\G'$ with density zero or at least
$\a'$. Moreover, define the \emph{cluster graph} $C$ as the graph
whose vertices are the $k$ non-exceptional clusters $V_{i}$, and
whose edges are the pairs of clusters between which there is nonzero
density in $\G'$. The minimum degree of $C$ is at least $\a'k$.
\end{lem}
This version of Szemer\'edi's regularity lemma is just the ``degree
form'' of \cite[Theorem~1.10]{KS96}, plus a straightforward claim
about minimum degree. Our statement follows directly from \cite[Proposition~9]{KOT05}.

We now give some simple lemmas about $\left(\varepsilon,\delta\right)$-denseness.
Note that $\left(\varepsilon,\delta\right)$-denseness is basically
a one-sided version of $\varepsilon$-regularity that is more convenient
in proofs about super-regularity. In particular, an $\varepsilon$-regular
pair with density $\delta$ is $\left(\varepsilon,\delta-\varepsilon\right)$-dense.
Here and in later sections, most of the theorems about $\left(\varepsilon,\delta\right)$-dense
pairs correspond to analogous theorems for $\varepsilon$-regular
pairs with density about $\delta$.
\begin{lem}
\label{lem:combine-clusters}Consider disjoint vertex sets $V^{1},V_{1}^{2},\dots,V_{r}^{2}$,
such that each $V_{i}^{2}$ is the same size and each $\left(V^{1},V_{i}^{2}\right)$
is $\left(\varepsilon,\delta\right)$-dense. Let $V^{2}=\bigcup_{i=1}^{r}V_{i}^{2}$;
then $\left(V^{1},V^{2}\right)$ is $\left(\varepsilon,\delta/r\right)$-dense.\end{lem}
\begin{proof}
Let $U^{\h}\subseteq V^{\h}$ with $\left|U^{\h}\right|\ge\varepsilon\left|V^{\h}\right|$.
Let $V_{i}^{2}$ be the cluster which has the largest intersection
with $U^{2}$, so we have $\left|U^{2}\cap V_{i}^{2}\right|\ge\varepsilon\left|V^{2}\right|/r=\varepsilon\left|V_{i}^{2}\right|$,
and there are therefore at least $\delta\left|U^{1}\right|\left|U^{2}\cap V_{i}^{2}\right|\ge\left(\delta/r\right)\left|U^{1}\right|\left|U^{2}\right|$
edges between $U^{1}$ and $U^{2}$.\end{proof}
\begin{lem}
\label{lem:robust-regular}Let $\left(V^{1},V^{2}\right)$ be an $\left(\varepsilon,\delta\right)$-super-regular
pair. Suppose we have $W^{\h}\supseteq V^{\h}$ with $\left|W^{\h}\right|\le\left(1+f\varepsilon\right)\left|V^{\h}\right|$,
and suppose that each vertex in $W^{\h}\backslash V^{\h}$ has at
least $\delta'\left|W^{\flip{\h}}\right|$ neighbours in $W^{\flip{\h}}$.
Then $\left(W^{1},W^{2}\right)$ is an $\left(\varepsilon',\delta'\right)$-super-regular
pair, where 
\[
\varepsilon'=\max\left\{ 2f,1+f\right\} \,\varepsilon,\quad\delta'=\min\left\{ 1/4,1/\left(1+f\varepsilon\right)\right\} \,\delta.
\]
\end{lem}
\begin{proof}
Suppose $U^{\h}\subseteq W^{\h}$ with $\left|U^{\h}\right|\ge\varepsilon'\left|W^{\h}\right|$.
Then $\left|U^{\h}\cap V^{\h}\right|\ge\varepsilon'\left|W^{\h}\right|-f\varepsilon\left|V^{\h}\right|\ge\varepsilon\left|V^{\h}\right|$
so there are at least $\delta\left|U^{1}\cap V^{1}\right|\left|U^{2}\cap V^{2}\right|$
edges between $U^{1}$ and $U^{2}$. But note that $\left|U^{\h}\cap V^{\h}\right|\ge\left|U^{\h}\right|-f\varepsilon\left|W^{\h}\right|\ge\left(1-\left(f\varepsilon\right)/\varepsilon'\right)\left|U^{\h}\right|\ge\left(1/2\right)\left|U^{\h}\right|$,
so $d\left(U^{1},U^{2}\right)\ge\delta/4$ and $\left(W^{1},W^{2}\right)$
is $\left(\varepsilon',\delta'\right)$-dense.

Next, note that each $v\in V^{\h}$ has $\delta\left|V^{\flip{\h}}\right|\ge\delta'\left|W^{\flip{\h}}\right|$
neighbours in $W^{\flip{\h}}$, and by assumption each $v\in W^{\h}\backslash V^{\h}$
has $\delta'\left|W^{\flip{\h}}\right|$ neighbours in $W^{\flip{\h}}$,
proving that $\left(W^{1},W^{2}\right)$ is $\left(\varepsilon',\delta'\right)$-super-regular.\end{proof}
\begin{lem}
\label{lem:regular-contains-superregular}Every $\left(\varepsilon,\delta\right)$-dense
pair $\left(V^{1},V^{2}\right)$ contains a $\left(\varepsilon/\left(1-\varepsilon\right),\delta-\varepsilon\right)$-super-regular
sub-pair $\left(W^{1},W^{2}\right)$, where $\left|W^{\h}\right|\ge\left(1-\varepsilon\right)\left|V^{\h}\right|$.
\end{lem}
\ref{lem:regular-contains-superregular} follows easily from the same
proof as \cite[Proposition~6]{BST08}.

Now we prove \ref{lem:partition-into-blobs}. First we need a lemma
about a decomposition into small stars.
\begin{lem}
\label{lem:star-cover}Let $G$ be an $n$-vertex graph with minimum
degree at least $\a n$. Then there is a spanning subgraph $S$ which
is a union of vertex-disjoint stars, each with at least two and at
most $1+1/\a$ vertices.\end{lem}
\begin{proof}
Let $S$ be a union of such stars with the maximum number of vertices.
Suppose there is a vertex $v$ uncovered by $S$. If $v$ has a neighbour
which is a center of one of the stars in $S$, and that star has fewer
than $1+\floor{1/\a}$ vertices, then we could add $v$ to that star,
contradicting maximality. (Here we allow either vertex of a 2-vertex
star to be considered the ``center''). Otherwise, if $v$ has a
neighbour $w$ which is a leaf of one of the stars in $S$, then we
could remove that leaf from its star and create a new 2-vertex star
with edge $vw$, again contradicting maximality. The remaining case
is where each of the (at least $\a n$) neighbours of $v$ is a center
of a star with $1+\floor{1/\a}>1/\a$ vertices. But these stars would
comprise more than $n$ vertices, which is again a contradiction.
We conclude that $S$ covers $G$, as desired.
\end{proof}

\begin{proof}[Proof of \ref{lem:partition-into-blobs}]
Apply our minimum degree form of Szemer\'edi's regularity lemma,
with some $\varepsilon'$ to be determined (which we will repeatedly
assume is sufficiently small). Consider a cover of the cluster graph
$C$ by stars $T_{1},\dots,T_{q}$ of size at most $1+1/\a'$, as
guaranteed by \ref{lem:star-cover}. Let $W_{i}^{1}$ be the center
cluster of $T_{i}$, and let $W_{i}^{2}$ be the union of the leaf
clusters of $T_{i}$ (for two-vertex stars, arbitrarily choose one
vertex as the ``leaf'' and one as the ``center''). Each edge of
the cluster graph $C$ corresponds to an $\left(\varepsilon',\a'/2\right)$-dense
pair, and each $W_{i}^{\h}$ is the union of at most $1/\a'$ clusters
$V_{i}$. By \ref{lem:combine-clusters}, with $\delta'=\left(\a'\right)^{2}/2$
each pair $\left(W_{i}^{1},W_{i}^{2}\right)$ is $\left(\varepsilon',\delta'\right)$-dense
in $\G'$ (therefore $\G$).

Apply \ref{lem:regular-contains-superregular} to obtain sets $V_{i}^{\h}\subseteq W_{i}^{\h}$
such that $\left|V_{i}^{\h}\right|=\left(1-\varepsilon'\right)\left|W_{i}^{\h}\right|$
and each $\left(V_{i}^{1},V_{i}^{2}\right)$ is $\left(2\varepsilon',\delta''\right)$-super-regular,
for $\delta''=\delta'/2$. Combining the exceptional cluster $V_{0}$
and all the $W_{i}^{\h}\backslash V_{i}^{\h}$, there are at most
$2\varepsilon'n$ ``bad'' vertices which are not part of a super-regular
pair.

Each bad vertex $v$ has at least $\a'n-2\varepsilon'n$ neighbours
to the clusters $V_{i}^{\h}$. The clusters $V_{i}^{\h}$ which contain
fewer than $\delta''\left|V_{i}^{\h}\right|$ of these neighbours,
altogether contain at most $\delta''n$ such neighbours. Each $V_{i}^{\h}$
has size at most $sn/\a'$, so for small $\varepsilon'$ there are
at least
\[
\frac{\a'n-2\varepsilon'n-\delta''n}{sn/\a'}\ge\left(\a'\right)^{2}/\left(2s\right)
\]
clusters $V_{i}^{\h}$ which have at least $\delta''\left|V_{i}^{\h}\right|$
neighbours of $v$. Pick one such $V_{i}^{\h}$ uniformly at random,
and put $v$ in $V_{i}^{\flip{\h}}$ (do this independently for each
bad $v$). By a concentration inequality, after this procedure a.a.s.
at most $2\left(2\varepsilon'n\right)/\left(\left(\a'\right)^{2}/\left(2s\right)\right)\le\left(8\varepsilon'/\left(\a'\right)^{2}\right)\left|W_{i}^{\h}\right|\le\left(9\varepsilon'/\left(\a'\right)^{2}\right)\left|V_{i}^{\h}\right|$
bad vertices have been added to each $V_{i}^{\h}$. By \ref{lem:robust-regular},
for small $\varepsilon'$ each $\left(V_{i}^{1},V_{i}^{2}\right)$
is now $\left(O\left(\varepsilon'\right),\delta''/4\right)$-super-regular.
To conclude, let $\delta=\delta''/4$ and $\rho=2\a'$, and choose
some $\varepsilon'$ small relative to $\varepsilon$.
\end{proof}
\global\long\def\bi{\mathbf{i}}

\global\long\def\bj{\mathbf{j}}

\global\long\def\br{\mathbf{r}}

We apply \ref{lem:partition-into-blobs} to our graph $\G$, with
$\varepsilon_{1}=\varepsilon_{1}\left(\delta\right)$ to be determined.
For $\bi=\left(i,\h\right)$, let $V_{\bi}=V_{i}^{\h}$, and let $\left|V_{\bi}\right|=s_{\bi}n$.

\subsubsection{Embedding a subforest of $\protect\T$\label{sub:embed-forest}}

In this section we will start to embed $\T$ into $\G\cup\R$. Just
as in \ref{sub:many-leaves}, we will use the decomposition of $\R$
into ``phases'' $\R_{1},\R_{2},\R_{3},\R_{4}$ (but we start with
$\R_{2}$ for consistency with the section numbering).

Recall that a \emph{bare path} in $\T$ is a path $P$ such that every
vertex of $P$ has degree exactly two in $\T$. Proceeding with the
proof outline, we need the fact that $\T$ is almost entirely composed
of bare paths, as is guaranteed by the following lemma due to one
of the authors (\cite[Lemma~2.1]{Kri10}).
\begin{lem}
\label{lem:bare-paths-if-no-leaves}Let $\T$ be a tree on $n$ vertices
with at most $\ell$ leaves. Then $\T$ contains a collection of at
least $\left(n-\left(2\ell-2\right)\left(k+1\right)\right)/\left(k+1\right)$
vertex-disjoint bare paths of length $k$ each.
\end{lem}
\global\long\def\r#1{r\left(#1\right)}

\global\long\def\x{x}

\global\long\def\y{y}

In our case $\ell=\lambda n$. If we choose $k$ large enough and
choose $\lambda$ small enough relative to $k$, then $\T$ contains
a collection of $n/\left(2\left(k-1\right)\right)$ disjoint bare
$k$-paths. (The definite value of $k$ will be determined later;
it will be odd and depend on $\a$ but not $\varepsilon$). If we
delete the interior vertices of these paths, then we are left with
a forest $\F$ on $n/2$ vertices.

Now, embed $\F$ into the random graph $\R_{2}$ (also, with some
abuse of notation, denote this embedded subgraph by $\F$). There
are $n/\left(2\left(k-1\right)\right)$ ``special pairs'' of ``special
vertices'' of $\F\subseteq\R_{2}$ that need to be connected with
$k$-paths of non-special vertices. We call such connecting paths
``special paths''. Let $X\subseteq V$ be the set of special vertices
and let $W=V\backslash\F$ be the set of ``free'' vertices. By symmetry
we can assume 
\[
X\,\cup\,\F\backslash X\,\cup\,W
\]
is a uniformly random partition of $V$ into parts of sizes $n/\left(k-1\right)$,
$n/2-n/\left(k-1\right)$ and $n/2$ respectively. We can also assume
that the special pairs correspond to a uniformly random partition
of $X$ into pairs. Note that $2\left|W\right|=\left(k-1\right)\left|X\right|$.

Let $X_{\bi}=V_{\bi}\cap X$ and $W_{\bi}=V_{\bi}\backslash F$ be
the set of free vertices remaining in $V_{\bi}$. By \ref{lem:hypergeometric-concentration}
and the union bound, a.a.s. each $\left|W_{\bi}\right|\sim s_{\bi}n/2$
and $\left|X_{\bi}\right|\sim s_{\bi}n/\left(k-1\right)$.

Here and in future parts of the proof, it is critical that after we
take certain subsets of our super-regular pairs, we maintain super-regularity
or at least density (albeit with weaker $\varepsilon$ and $\delta$).
We will ensure this in each situation by appealing to one of the following
two lemmas.
\begin{lem}
\label{lem:large-preserves-dense}Suppose $\left(V^{1},V^{2}\right)$
is $\left(\varepsilon,\delta\right)$-dense, and let $W^{\h}\subseteq V^{\h}$
with $\left|W^{\h}\right|\ge\gamma\left|V^{\h}\right|$. Then $\left(W^{1},W^{2}\right)$
is $\left(\varepsilon/\gamma,\delta\right)$-dense.\end{lem}
\begin{proof}
If $U^{\h}\subseteq W^{\h}$ with $\left|U^{\h}\right|\ge\left(\varepsilon/\gamma\right)\left|W^{\h}\right|$,
then $\left|U^{\h}\right|\ge\varepsilon\left|V^{\h}\right|$. The
result follows from the definition of $\left(\varepsilon,\delta\right)$-denseness.\end{proof}
\begin{lem}
\label{lem:random-preserves-superregular}Fix $\delta,\varepsilon'>0$.
There is $\varepsilon=\varepsilon\left(\varepsilon'\right)>0$ such
that the following holds. Suppose $\left(V^{1},V^{2}\right)$ is $\left(\varepsilon,\delta\right)$-dense,
let $n_{\h}$ satisfy $\left|V^{\h}\right|\ge n_{\h}=\Omega$$\left(n\right)$
(not necessarily uniformly over $\varepsilon$ and $\delta$) and
for each $\h$ let $W^{\h}\subseteq V^{\h}$ be a uniformly random
subset of size $n_{\h}$. Then $\left(W^{1},W^{2}\right)$ is a.a.s.
$\left(\varepsilon',\delta/2\right)$-dense. If moreover $\left(V^{1},V^{2}\right)$
was $\left(\varepsilon,\delta\right)$-super-regular then $\left(W^{1},W^{2}\right)$
is a.a.s. $\left(\varepsilon',\delta/2\right)$-super-regular.
\end{lem}
\ref{lem:random-preserves-superregular} is a simple consequence of
a highly nontrivial (and much more general) result proved by Gerke,
Kohayakawa, R\"odl and Steger in \cite{GKRS07}, which shows that
small random subsets inherit certain regularity-like properties with
very high probability. The notion of ``lower regularity'' in that
paper essentially corresponds to our notion of $\left(\varepsilon,\delta\right)$-density.
\begin{proof}[Proof of \ref{lem:random-preserves-superregular}]
The fact that $\left(W^{1},W^{2}\right)$ is a.a.s. $\left(\varepsilon',\delta/2\right)$-dense
for suitable $\varepsilon$, follows from two applications of \cite[Theorem~3.6]{GKRS07}.

It remains to consider the degree condition for super-regularity.
For each $v\in W^{\h}$, the number of neighbours of $v$ in $W^{\flip{\h}}$
is hypergeometrically distributed, with expected value at least $\delta n_{\flip{\h}}$.
Since $n_{\flip h}=\Omega\left(n\right)$, \ref{lem:hypergeometric-concentration} and the union
bound immediately tells us that a.a.s. each such $v$ has $\delta n_{\flip h}/2$
such neighbours, as required.
\end{proof}
Let $\flip{\left(i,\h\right)}=\left(i,\flip{\h}\right)$, and let
$\left[\left(i,\h\right)\right]=i$. It is an immediate consequence
of \ref{lem:random-preserves-superregular} that each $\left(X_{\bi},W_{\flip{\bi}}\right)$
and $\left(W_{\bi},W_{\flip{\bi}}\right)$ are $\left(\varepsilon_{2},\delta_{2}\right)$-super-regular,
where $\delta_{2}=\delta_{1}/2$ and $\varepsilon_{2}$ can be made
arbitrarily small by choice of $\varepsilon_{1}$.

Now that we have embedded $\F$, we no longer care about the vertices
used to embed $\F\backslash X$; they will never be used again. It
is convenient to imagine, for the duration of the proof, that instead
of embedding parts of $\T$, we are ``removing'' vertices from $\G\cup\R$,
gradually making the remaining graph easier to deal with. So, update
$n$ to be the number of vertices not used to embed $\F\backslash X$
(previously $\left(1/2+1/\left(k-1\right)\right)n$), and update $s_{\bi}$
to satisfy $\left|W_{\bi}\right|=s_{\bi}n$. Note that (asymptotically
speaking) this just rescales the $s_{\bi}$ so that they sum to $\left(\left(k-1\right)/\left(k+1\right)\right)$.
Therefore the variation between the $s_{\bi}$ is still bounded independently
of $\varepsilon$; we can increase $\rho$ slightly so that each $\left|s_{\bi}\right|/\left|s_{\bj}\right|\le\rho$.
We then have $\left|X_{\bi}\right|\sim2s_{\bi}n/\left(k-1\right)$
for each $\bi$, and there are $n/\left(k+1\right)$ special pairs.

\subsubsection{\label{sub:fix-endpoints}Fixing the endpoints of the special paths}

We will eventually need to start finding and removing special paths
between our special pairs. For this, we would like each of the special
pairs to be ``between'' partner clusters $V_{\bi},V_{\flip{\bi}}$,
so that we can take advantage of the super-regularity in $\G$. Therefore,
for every special pair $\left\{ \x,\y\right\} $, we will find very
short (length-2) paths from $\x$ to some vertex $\x'\in W_{\br}$
and from $\y$ to some $\y'\in W_{\flip{\br}}$, for some $\br$.
All these short paths will be disjoint. Our short paths effectively
``move'' the special pair $\left\{ \x,\y\right\} $ to $\left\{ \x',\y'\right\} $:
to complete the embedding we now need to find length-$\left(k-4\right)$
special paths connecting the new special pairs.

Note that there is a.a.s. an edge between any two large vertex sets
in the random graph $\R_{3}$, as formalized below.
\begin{lem}
\label{lem:edge-between-large-sets}For any $\varepsilon>0$, there
is $\c=\c\left(\varepsilon\right)$ such that the following holds.
In a random graph $\R\in\GG\left(n,\c/n\right)$, there is a.a.s.
an edge between any two (not necessarily disjoint) vertex subsets
$A,B$ of size $\varepsilon n$.
\end{lem}
\ref{lem:edge-between-large-sets} is a standard result and can be
proved in basically the same way as \ref{lem:bipartite-perturbed-stars}.
We include a proof for completeness.
\begin{proof}
For any such subsets $A,B$, choose disjoint $A'\subseteq A$ and
$B'\subseteq B$ of size $\varepsilon n/2$. The probability there
is no edge between $A$ and $B$ is less than the probability there
is no edge between $A'$ and $B'$, which is $\left(1-\c/n\right)^{\left(\varepsilon n/2\right)^{2}}\le e^{-\c\varepsilon^{2}n/2}$.
There are at most $2^{2n}$ choices of $A$ and $B$ so for large
$\c$, by the union bound there is a.a.s. an edge between any such
pair of sets.
\end{proof}
It is relatively straightforward to greedily find suitable short paths
in $\G\cup\R_{3}$ using the minimum degree condition in the super-regular
pairs $\left(X_{\bi},W_{\flip{\bi}}\right)$ and \ref{lem:edge-between-large-sets}.
However we need to be careful to find and remove our paths in such
a way that afterwards we still have super-regularity between certain
subsets. We will accomplish this by setting aside a random subset
of each $W_{\br}$ from which the $\x'$, $\y'$ (and no other vertices
in our length-2 paths) will be chosen, then using the symmetry argument
of \ref{rem:symmetry} and appealing to \ref{lem:random-preserves-superregular}.
The details are a bit involved, as follows.

Uniformly at random partition each $W_{\bi}$ into sets $W_{\bi}^{\left(1\right)}$
and $W_{\bi}^{\left(2\right)}$ of size $\left|W_{\bi}\right|/2$.
We will take the $\x'$,$\y'$ from the $W_{\bi}^{\left(2\right)}$s
and we will take the intermediate vertices in our length-2 paths from
the $W_{\bi}^{\left(1\right)}$s. Note that in $\G$, a.a.s. each
$\x\in X_{\bi}$ has at least $\left(\delta_{2}/3\right)s_{\flip{\bi}}n$
neighbours in $W_{\bi}^{\left(1\right)}$. This follows from \ref{lem:hypergeometric-concentration} applied to the
hypergeometric random variables $\left|N_{\G}\left(\x\right)\cap W_{\flip{\bi}}^{\left(1\right)}\right|$
(which have mean at least $\delta_{2}s_{\flip{\bi}}n/2$ by super-regularity),
plus the union bound. Arbitrarily choose an ordering of the special
pairs, and choose some ``destination'' index $\br$ for each special
pair, in such a way that each $\br$ is chosen for a $\left(1+o\left(1\right)\right)/\left(2q\right)$
fraction of the special pairs. (Recall that $q$ is the number of
pairs of partner clusters). We will find our length-2 paths greedily,
by repeatedly applying the following argument.

Suppose we have already found length-2 paths for the first $i-1$
special pairs, and let $S$ be the set of vertices in these paths.
Let $\left\{ \x,\y\right\} $ be the $i$th special pair, with $\x\in X_{\bi}$
and $\y\in X_{\bj}$. Let $\br$ be the destination index for this
special pair. We claim that if $k$ is large then $\left|N_{\G}\left(\x\right)\cap\left(W_{\flip{\bi}}^{\left(1\right)}\backslash S\right)\right|=\Omega\left(n\right)$
and $\left|W_{\br}^{\left(2\right)}\backslash S\right|=\Omega\left(n\right)$.
It will follow from \ref{lem:edge-between-large-sets} that there
is a.a.s. an edge between $N_{\G}\left(\x\right)\cap\left(W_{\flip{\bi}}^{\left(1\right)}\backslash S\right)$
and $W_{\br}^{\left(2\right)}\backslash S$ in $\R_{3}$. This gives
a length-2 path between $\x$ and some $\x'\in W_{\br}^{\left(2\right)}$
disjoint to the paths so far, in $\G\cup\R_{3}$. By identical reasoning
there is a disjoint length-2 path between $\y$ and some $\y'\in W_{\flip{\br}}^{\left(2\right)}$
passing through $W_{\flip{\bj}}^{\left(1\right)}$.

To prove the claims in the preceding argument, note that we always
have $\left|W_{\flip{\bi}}^{\left(1\right)}\cap S\right|\le\left|X_{\bi}\right|$
and $\left|W_{\br}^{\left(2\right)}\cap S\right|\le\left(1+o\left(1\right)\right)\left|X\right|/\left(2q\right)$.
Let $s=\left(\left(k-1\right)/\left(k+1\right)\right)/2q$ be the
average $s_{\bi}$, and recall that $\rho$ controls the relative
sizes of the $s_{\bi}$. For large $k$,
\[
\left|N_{\G}\left(\x\right)\cap\left(W_{\flip{\bi}}^{\left(1\right)}\backslash S\right)\right|\ge\frac{\delta_{2}s_{\flip{\bi}}}{3}n-\left|X_{\bi}\right|\sim\frac{\delta_{2}s_{\flip{\bi}}}{3}n-\frac{2s_{\bi}}{k-1}n\ge\left(\frac{\delta_{2}}{3}-\frac{2\rho}{k-1}\right)s_{\flip{\bi}}n=\Omega\left(n\right),
\]
and
\[
\left|W_{\br}^{\left(2\right)}\backslash S\right|\ge\frac{\left|W_{\bi}\right|}{2}-\left(1+o\left(1\right)\right)\frac{\left|X\right|}{2q}\sim\frac{s_{\bi}}{2}n-\frac{2s}{k-1}n\ge\left(\frac{1}{2}-\frac{2\rho}{k-1}\right)s_{\bi}n=\Omega\left(n\right),
\]
as required. We emphasize that since $\delta_{2}$ and $\rho$ are
independent of $\varepsilon$ (depending only on $\alpha$), we can
also choose $k$ independent of $\varepsilon$.

After we have found our length-2 paths, let $X_{\bi}'\subseteq W_{\bi}$
be the set of ``new special vertices'' $\x',\y'$ in $W_{\bi}$.
Note that we can in fact assume that each $X_{\bi}'$ is a uniformly
random subset of $W_{\bi}^{\left(2\right)}$ of its size (which is
$\left|X\right|/\left(2q\right)$). This can be seen by a symmetry
argument, as per \ref{rem:symmetry}. Because here the situation is
a bit complicated, we include some details. Condition on the random
sets $W_{\bi}^{\left(2\right)}$. For each $\bi$, the vertices in
$X_{\bi}'$ are each chosen so that they are adjacent in $\R_{3}$
to some vertex in one of the $W_{\bj}^{\left(1\right)}$. Note that
the distribution of $\R_{3}$ is invariant under permutations of $W_{\bi}^{\left(2\right)}$.
We can choose such a vertex permutation $\pi$ uniformly at random;
then $\pi\left(X_{\bi}'\right)$ is a uniformly random subset of $W_{\br}^{\left(2\right)}$
of its size, which has the right adjacencies in $\pi\left(\R_{3}\right)$
(which has the same distribution as $\R_{3}$).

Now we claim that even after removing the vertices of our length-2
paths we a.a.s. maintain some important properties of the clusters.
Let $W_{\bi}'\subseteq W_{\bi}\backslash X_{\bi}'$ be the subset
of $W_{\bi}$ that remains after the vertices in the length-2 paths
are deleted. 
\begin{claim}
\label{claim:move-ok}Each $\left(X_{\bi}',W_{\flip{\bi}}'\right)$
and $\left(W_{\bi}',W_{\flip{\bi}}'\right)$ are a.a.s. $\left(\rho\left(k-1\right)\varepsilon_{2}/\left(1-\delta_{2}/2\right),\delta_{2}/2\right)$-super-regular,
satisfying 
\[
\left(2\rho\right)^{-1}\le\frac{k-1}{2}\cdot\frac{\left|X_{\bi}'\right|}{\left|W_{\bi}'\right|}\le2\rho.
\]
\end{claim}
\begin{proof}
Since each $W_{\bi}^{\left(2\right)}$ is a uniformly random subset
of $W_{\bi}$, we can assume each $X_{\bi}'$ is a uniformly random
subset of $W_{\bi}$ of its size. So, each $\left|N_{\G}\left(w\right)\cap X_{\bi}'\right|$
is hypergeometrically distributed with mean at least $\delta_{2}\left|X_{\bi}'\right|$.
By \ref{lem:hypergeometric-concentration}
and the union bound, in $\G$ a.a.s. each $w\in W_{\flip{\bi}}$ has
at least $\left(\delta_{2}/2\right)\left|X_{\bi}'\right|$ neighbours
in $X_{\bi}'$. Next, note that each $\left|X_{\bi}'\right|\sim2sn/\left(k-1\right)$.
Since $s/s_{\flip{\bi}}\ge1/\rho$, we have $\rho\left(k-1\right)\left|X_{\bi}'\right|\ge\left|W_{\flip{\bi}}\right|$
and by \ref{lem:large-preserves-dense} each $\left(X_{\bi}',W_{\flip{\bi}}\right)$
is $\left(\rho\left(k-1\right)\varepsilon_{2},\delta_{2}\right)$-dense.
It follows that each $\left(X_{\bi}',W_{\flip{\bi}}\right)$ is $\left(\rho\left(k-1\right)\varepsilon_{2},\delta_{2}/2\right)$-super-regular.

Now, for large $k$ note that $\left|W_{\bi}\backslash W_{\bi}'\right|\le\left(\delta_{2}/2\right)\left|W_{\bi}\right|$.
Indeed, recalling the choice of vertices for the length-2 paths, note
that 
\[
\frac{\left|W_{\bi}\backslash W_{\bi}'\right|}{\left|W_{\bi}\right|}=\frac{\left|X_{\flip{\bi}}\right|+\left|X\right|/2q}{\left|W_{\bi}\right|}\sim\frac{\frac{2}{k-1}\left(s_{\flip{\bi}}+s\right)}{s_{\bi}}\le\frac{4\rho}{k-1},
\]
which is smaller than $\delta_{2}/2$ for large $k$. For $v\in W_{\bi}$
we have $\left|N_{\G}\left(v\right)\cap W_{\flip{\bi}}\right|\ge\delta_{2}\left|W_{\flip{\bi}}\right|$
by super-regularity, so 
\[
\left|N_{\G}\left(v\right)\cap W_{\flip{\bi}}'\right|\ge\left|N_{\G}\left(v\right)\cap W_{\flip{\bi}}\right|-\left|W_{\bi}\backslash W_{\bi}'\right|\ge\left(\delta_{2}/2\right)\left|W_{\flip{\bi}}\right|\ge\left(\delta_{2}/2\right)\left|W_{\flip{\bi}}'\right|.
\]
Combining this with \ref{lem:large-preserves-dense}, a.a.s. each
$\left(W_{\bi}',W_{\flip{\bi}}'\right)$ is $\left(\varepsilon_{2}/\left(1-\delta_{2}/2\right),\delta_{2}/2\right)$-super-regular
and each $\left(X_{\bi}',W_{\flip{\bi}}'\right)$ is $\left(\rho\left(k-1\right)\varepsilon_{2}/\left(1-\delta_{2}/2\right),\delta_{2}/2\right)$-super-regular.

Finally, the claims about the relative sizes of the $W_{\bi}'$, $X_{\bi}'$
for large $k$ are simple consequences of the asymptotic expressions
\[
\left|W_{\bi}'\right|\sim s_{\bi}n-\frac{2}{k-1}\left(s_{\flip{\bi}}+s\right)n,\quad\left|X_{\bi}'\right|\sim\frac{2s}{k-1}n.\qedhere
\]

\end{proof}
Now we can redefine each $X_{\bi}$ to be the set of new special vertices
taken from $W_{\bi}$, and we can update the $W_{\bi}$ by removing
all the vertices used in our length-2 paths (that is, set $X_{\bi}=X_{\bi}'$
and $W_{\bi}=W_{\bi}'$). All the special pairs are between some $X_{\bi}$
and $X_{\flip{\bi}}$; let $Z_{\left[\bi\right]}=Z_{\left[\flip{\bi}\right]}$
be the set of such pairs (recall that $\left[\left(i,\h\right)\right]=i$).
Update the sets $X=\bigcup_{\bi}X_{\bi}$ and $W=\bigcup_{\bi}W_{\bi}$,
and let $Z=\bigcup_{i}Z_{i}$ be the set of all special pairs, so
$\left|Z\right|=\left|X\right|/2$.

We also update all the other variables as before: update $n$ to be
the number of vertices in $X\cup W$ (previously $n-4n/\left(k+1\right)$),
and redefine the $s_{\bi}$ to still satisfy $\left|W_{\bi}\right|=s_{\bi}n$.
Also, update $k$ to be the new required length of special paths (previously
$k-4$), so we still have $2\left|W\right|=\left(k-1\right)\left|X\right|$.

Finally, let $\varepsilon_{3}=\rho\left(k-1\right)\varepsilon_{2}/\left(1-\delta_{2}/2\right)$
and $\delta_{3}=\delta_{2}/2$ and double the value of $\rho$. By
\ref{claim:move-ok}, each $\left(X_{\bi},W_{\flip{\bi}}\right)$
and $\left(W_{\bi},W_{\flip{\bi}}\right)$ are then $\left(\varepsilon_{3},\delta_{3}\right)$-super-regular,
and each 
\[
\rho^{-1}\le\frac{k-1}{2}\cdot\frac{\left|X_{\bi}\right|}{\left|W_{\bi}\right|}\le\rho.
\]

\subsubsection{Adjusting the relative sizes of the clusters\label{sub:cluster-adjustment}}

The next step is to use the random edges in $\R_{4}$ to start removing
special paths. Our objective is to do this in such a way that after
we have removed the vertices in those paths from the $X_{\bi}$ and
$W_{\bi}$,  we will have $2\left|W_{\bi}\right|=\left(k-1\right)\left|X_{\bi}\right|$
for each $\bi$. We require this to apply the blow-up lemma in the
final step.

The way we will remove special paths is with \ref{lem:adjusting-paths},
to follow.

A ``special sequence'' is a sequence of vertices $\x,w_{1},\dots,w_{k-1},\y$,
such that $\left\{ \x,\y\right\} $ is a special pair and each $w_{j}$
is in some $W_{\bi}$. For any special path we might find, the sequence
of vertices in that path will be special. In fact, the special sequences
are precisely the allowed sequences of vertices in special paths.

We also define an equivalence relation on special sequences: two special
sequences $\x,w_{1},\dots,w_{k-1},\y$ and and $\x',w_{1}',\dots,w_{k-1}',\y'$
are equivalent if $\x$ and $\x'$ are in the same cluster $X_{\bi}$
(therefore $\y,\y'\in X_{\flip{\bi}}$), and if $w_{j}$ and $w_{j}'$
are in the same cluster $W_{\bi_{j}}$ for all $j$. A ``template''
is an equivalence class of special sequences, or equivalently a sequence
of clusters $X_{\bi},W_{\bi_{1}},\dots,W_{\bi_{k-1}},X_{\flip{\bi}}$
(repetitions permitted).

The idea is that we can specify (a multiset of) desired templates,
and provided certain conditions are satisfied we can find special
paths with these templates, using the remaining vertices in $\G\cup\R_{4}$.
Say a special sequence $\x,w_{1},\dots,w_{k-1},\y$ with $\x\in X_{\bi}$
and $\y\in X_{\flip{\bi}}$ is ``good'' if $w_{1}\in W_{\flip{\bi}}$
and $w_{k-1}\in W_{\bi}$. We make similar definitions for good special
paths and good templates. It is easier to find good special paths
because we can take advantage of super-regularity at the beginning
and end of the paths.

Recall that $\R_{4}\in\GG\left(n,\c_{4}/n\right)$; we still have
the freedom to choose large $\c_{4}$ to make it possible to find
our desired special paths in $\G\cup\R_{4}$. Recall from previous
sections that $s$ and $\rho$ control the sizes of the clusters,
that $Q$ controls the number of clusters, and that each $\left(X_{\bi},W_{\flip{\bi}}\right)$
and $\left(W_{\bi},W_{\flip{\bi}}\right)$ are $\left(\varepsilon_{3},\delta_{3}\right)$-super-regular,
where we can make $\varepsilon_{3}$ as small as we like relative
to $\alpha$.
\begin{claim}
\label{lem:adjusting-paths}For any $\gamma>0$, if $\varepsilon_{3}$
is sufficiently small relative to $\gamma$ and $\delta_{3}$, there
is $\c_{4}=\c_{4}\left(s,\rho,\delta_{3},\gamma,Q,k\right)$ such
that the following holds.

Suppose we specify a multiset of templates (each of which we interpret
as sequences of clusters), in such a way that each $X_{\bi}$ (respectively,
each $W_{\bi}$) appears in these templates at most $\left(1-\gamma\right)\left|X_{\bi}\right|$
times (respectively, at most $\left(1-\gamma\right)\left|W_{\bi}\right|$
times). Then, we can a.a.s. find special paths in $\G\cup\R_{4}$
with our desired templates in such a way that after the removal of
their vertices, each $\left(X_{\bi},W_{\flip{\bi}}\right)$ and $\left(W_{\bi},W_{\flip{\bi}}\right)$
are still $\left(\varepsilon_{3}/\gamma,\delta_{3}/4\right)$-super-regular.
\end{claim}
Note that our condition on the templates says precisely that at least
a $\gamma$-fraction of each $X_{\bi}$ and $W_{\bi}$ should remain
after deleting the vertices of the special paths.

We will prove \ref{lem:adjusting-paths} with a simpler lemma.
\begin{lem}
\label{lem:adjusting-paths-more-basic}For any $\delta,\xi>0$ and
any integer $k>2$, there are $\varepsilon\left(\delta,\xi\right)>0$
and $\c=\c\left(\delta,\xi,k\right)$ such that the following holds.

Let $\G$ be a graph on the vertex set $\range{\left(k+1\right)n}$,
together with a vertex partition into disjoint clusters $X,W_{1}\dots,W_{k-1},Y$,
such that $\left|X\right|=\left|Y\right|=n$ and each $\left|W_{i}\right|=n$.
Suppose that $\left(X,W_{1}\right)$ and $\left(W_{k-1},Y\right)$
are $\left(\varepsilon,\delta\right)$-dense and that each $\x\in X$
is bijectively paired with a vertex $\y\in Y$, comprising a \emph{special
pair} $\left(\x,\y\right)$. Let $\R\in\GG\left(\left(k+1\right)n,\c/n\right)$.
Then, in $\G\cup\R$ we can a.a.s. find $\left(1-\xi\right)n$ vertex-disjoint
paths running through $X,W_{1},\dots,W_{k-1},Y$ in that order, each
of which connects the vertices of a special pair (these are \emph{special
paths}).\end{lem}
\begin{proof}[Proof of \ref{lem:adjusting-paths-more-basic}]
If $\c$ is large, in $\R$ we can a.a.s. find an edge between any
subsets $W_{i}'\subseteq W_{i}$ and $W_{i+1}'\subseteq W_{i+1}$
with $\left|W_{i}'\right|,\left|W_{i+1}'\right|\ge\delta\xi n/\left(k-2\right)$,
for any $1\le i<k-1$. We can prove this fact with basically the same
argument as in the proof of \ref{lem:edge-between-large-sets}. In
fact, it immediately follows that a.a.s. between any such $W_{i}'$
and $W_{i+1}'$ there is a matching with more than $\min\left\{ \left|W_{i}'\right|,\left|W_{i+1}'\right|\right\} -\delta\xi n/\left(k-2\right)$
edges, in $\R$. This is because in a maximal matching, the subsets
of unmatched vertices in $W_{i}'$ and $W_{i+1}'$ have no edge between
them.

It follows that (a.a.s.) if we choose any subsets $W_{i}'\subseteq W_{i}$
($1\le i\le k-1$) with $\left|W_{1}'\right|=\dots=\left|W_{k-1}'\right|=\delta\xi n$,
there is a path in $\R$ running through the $W_{1}',\dots,W_{k-1}'$
in that order. Indeed, we have just proven there is a matching of
size greater than $\left|W_{i}'\right|-\delta\xi n/\left(k-2\right)$
between each $W_{i}'$ and $W_{i+1}'$; the union of such matchings
must include a suitable path.

Now, consider a $\xi n$-element subset $Z'$ of the special pairs,
and consider a $\xi n$-element subset $W_{i}'$ of each $W_{i}$.
We would like to find a special path using one of the special pairs
in $Z$ and a vertex from each $W_{i}'$. If we could (a.a.s.) find
such a special path for all such subsets $Z'$ and $W_{i}'$ then
we would be done, for we would be able to choose our desired $\left(1-\xi\right)n$
special paths greedily. That is, given a collection of fewer than
$\left(1-\xi\right)n$ disjoint special paths running through the
clusters, we could find an additional disjoint special path with the
leftover vertices.

Let $X'\subseteq X$ and $Y'\subseteq Y$ be the subsets containing
the vertices in the special pairs in $Z'$. By $\left(\varepsilon,\delta\right)$-denseness
(with $\varepsilon\le\xi$) there are at most $\varepsilon n$ vertices
in $X'$ (respectively, in $Y'$) with fewer than $\delta\xi n$ neighbours
in $W_{1}'$ (respectively, in $W_{k-1}'$), in $\G$. So, if $\varepsilon<\xi/2$,
then there must be a special pair $\left(\x,\y\right)$ such that
$\left|N_{G}\left(\x\right)\cap W_{1}'\right|,\left|N_{G}\left(\y\right)\cap W_{k-1}'\right|\ge\delta\xi n$.
By the discussion at the beginning of the proof, there is a.a.s. a
path in $\R$ running through $N_{G}\left(\x\right)\cap W_{1}',W_{2}',\dots,W_{k-2}',N_{G}\left(\y\right)\cap W_{k-1}'$
in that order. Combining this path with $\x$ and $\y$ gives us a
special path between $\x$ and $\y$, as desired.
\end{proof}
A corollary of \ref{lem:adjusting-paths-more-basic} is that we can
find special paths corresponding to a single template:
\begin{cor}
\label{cor:adjusting-paths-basic}For any $b,\delta,\xi>0$ and any
integer $k>2$, there are $\varepsilon\left(\delta,\xi\right)>0$
and $\c=\c\left(b,\delta,\xi,k\right)$ such that the following holds.

Let $\G$ be a graph on some vertex set $\range N$ ($n\ge bN$),
together with a vertex partition into clusters. Consider a sequence
of clusters $X,W_{1},\dots,W_{k-1},Y$ (there may be repetitions in
the $W_{i}$s), such that $\left|X\right|,\left|Y\right|\ge n$ and
if $W_{i}$ appears $t_{i}$ times in the sequence then $\left|W_{i}\right|\ge t_{i}n$.
Suppose that $\left(X,W_{1}\right)$ and $\left(W_{k-1},Y\right)$
are $\left(\varepsilon,\delta\right)$-dense, and that each $\x\in X$
is bijectively paired with a vertex $\y\in Y$, comprising a \emph{special
pair} $\left(\x,\y\right)$. Let $\R\in\GG\left(N,\c/N\right)$. Then,
in $\G\cup\R$ we can a.a.s. find $\left(1-\xi\right)n$ vertex-disjoint
paths running through $X,W_{1},\dots,W_{k-1},Y$ in that order, each
of which connects the vertices of a special pair (these are \emph{special
paths}).\end{cor}
\begin{proof}
Uniformly at random choose $t_{i}$ disjoint $n$-vertex subsets of
each $W_{i}$. This gives a total of $k-1$ disjoint subsets; arrange
these into a sequence $W_{1}',\dots,W_{k-1}'$ in such a way that
$W_{i}'\subseteq W_{i}$. Also uniformly at random choose $n$-vertex
subsets $X'\subseteq X$ and of $Y'\subseteq Y$. By \ref{lem:random-preserves-superregular},
$\left(X',W_{1}'\right)$ and $\left(W_{k-1}',Y'\right)$ are a.a.s.
$\left(\varepsilon',\delta\right)$-dense, where we can make $\varepsilon'$
small by choice of $\varepsilon$. We can therefore directly apply
\ref{lem:adjusting-paths-more-basic}.
\end{proof}
We can now prove \ref{lem:adjusting-paths}.
\begin{proof}[Proof of \ref{lem:adjusting-paths}]
The basic idea is to split the clusters into many random subsets,
one for each good template, and to apply \ref{cor:adjusting-paths-basic}
many times. (Note that there are fewer than $Q^{k}=O\left(1\right)$
different templates). We also need to guarantee super-regularity in
what remains. For this, we set aside another random subset of each
cluster, which we will not touch (we used a similar idea in \ref{sub:fix-endpoints}).

If we are able to successfully find our special paths, those with
template $\tau$ will contain some number $t_{\bi}^{\tau}\left|W_{\bi}\right|$
of vertices from $W_{\bi}$. (Note that $t_{\bi}^{\tau}$ depends
only on the multiset of desired templates, and not on the specific
special paths we find). After removing the vertices in our found special
paths, there will be $l_{\bi}\left|W_{\bi}\right|=\left(1-\sum_{\tau}t_{\bi}^{\tau}\right)\left|W_{\bi}\right|\ge\gamma\left|W_{\bi}\right|$
vertices remaining in $W_{\bi}$. Similarly, the special paths with
good template $\tau$ will take some number $s_{i}^{\tau}\left|Z_{i}\right|$
of special pairs from each $Z_{i}$, leaving $r_{i}\left|Z_{i}\right|\ge\gamma\left|Z_{i}\right|$
special pairs.

Now, for some small $\xi>0$ to be determined, we want to partition
each $W_{\bi}$ into parts $W_{\bi}^{\tau}$ of size at least $t_{\bi}^{\tau}\left|W_{\bi}\right|/\left(1-\xi\right)$,
and a ``leftover'' part $W_{\bi}'$ of size $\left|W_{\bi}\right|-\sum_{\tau}\left|W_{\bi}^{\tau}\right|$.
Similarly, we want to partition each $Z_{i}$ into parts $Z_{i}^{\tau}$
of size at least $s_{i}^{\tau}\left|Z_{i}\right|/\left(1-\xi\right)$
and a leftover part $Z_{i}'$. Let $X_{\bi}^{\tau}\subseteq X_{\bi}$
and $X_{\bi}'\subseteq X_{\bi}$ contain the special vertices in the
special pairs in the $Z_{\left[\bi\right]}^{\tau}$ and $Z_{\left[\bi\right]}'$
respectively.

Choose the sizes of the $W_{\bi}^{\tau}$ and $Z_{i}^{\tau}$ such
that each $\left|W_{\bi}^{\tau}\right|\ge\xi sn/Q^{k}$ and $\left|Z_{i}^{\tau}\right|\ge\xi sn/\left(\left(k-1\right)Q^{k}\right)$,
and also each $\left|W_{\bi}'\right|\ge\left(l_{\bi}/2\right)\left|W_{\bi}\right|$
and $\left|Z_{i}'\right|\ge\left(r_{i}/2\right)\left|Z_{i}\right|$
(this will be possible if $\xi$ is small enough to satisfy $\left(1-\gamma\right)/\left(1-\xi\right)+\xi\le\left(1-\gamma/2\right)$).
Given these part sizes choose our partitions uniformly at random.
(Note that this means the marginal distribution of each part is uniform
on subsets of its size).

By \ref{lem:random-preserves-superregular}, for any $\varepsilon'>0$
we can choose $\varepsilon_{3}$ such that each $\left(X_{\bi}^{\tau},W_{\flip{\bi}}^{\tau}\right)$
is a.a.s. $\left(\varepsilon',\delta_{3}/2\right)$-dense. Moreover
a.a.s. each $w\in W_{\flip{\bi}}$ has at least $\left(\delta_{3}/2\right)\left|W_{\bi}'\right|$
neighbours in $W_{\bi}'$ and at least $\left(\delta_{3}/2\right)\left|X_{\bi}'\right|$
neighbours in $X_{\bi}'$, and each $x\in X_{\bi}$ has at least $\left(\delta_{3}/2\right)\left|W_{\flip{\bi}}'\right|$
neighbours in $W_{\flip{\bi}}'$. We can find our desired special
pairs with good template $\tau$ using the clusters $X_{\bi}^{\tau}$,
$W_{\bi}^{\tau}$ and \ref{cor:adjusting-paths-basic}, provided $\varepsilon'$
is small (in the notation of \ref{cor:adjusting-paths-basic}, we
need $\varepsilon'\le\varepsilon\left(\delta_{3}/2,\xi\right)$).

After deleting the vertices in these special paths, by \ref{lem:large-preserves-dense}
each $\left(X_{\bi},W_{\flip{\bi}}\right)$ and $\left(W_{\bi},W_{\flip{\bi}}\right)$
are $\left(\varepsilon_{3}/\gamma,\delta_{3}\right)$-dense. By the
condition on the sizes of the $W_{\bi}'$s and $Z_{i}'$s and the
bounds on the neighbourhood sizes in the last paragraph, each $w\in W_{\flip{\bi}}$
now has at least $\left(\delta_{3}/4\right)\left|W_{\bi}\right|$
neighbours in $W_{\bi}$ and at least $\left(\delta_{3}/4\right)\left|X_{\bi}\right|$
neighbours in $X_{\bi}$, and each $x\in X_{\bi}$ has at least $\left(\delta_{3}/4\right)\left|W_{\flip{\bi}}\right|$
neighbours in $W_{\flip{\bi}}$. So each $\left(X_{\bi},W_{\flip{\bi}}\right)$
and $\left(W_{\bi},W_{\flip{\bi}}\right)$ are $\left(\varepsilon_{3}/\gamma,\delta_{3}/4\right)$-super-regular,
as required.
\end{proof}
Having proved \ref{lem:adjusting-paths}, we now just need to show
that there are suitable good templates so that after finding special
paths with those templates and removing their vertices, we end up
with $2\left|W_{\bi}\right|=\left(k-1\right)\left|X_{\bi}\right|$
for each $\bi$.

Let
\[
m_{i}=\min\left\{ \left|Z_{i}\right|,\,2\left|W_{\left(i,1\right)}\right|/\left(k-1\right),\,2\left|W_{\left(i,2\right)}\right|/\left(k-1\right)\right\} .
\]
Impose that $k$ is odd; we want to leave $L_{i}^{Z}:=\floor{m_{i}/2}$
special pairs in each $Z_{i}$ and $L_{i}^{W}:=\left(\left(k-1\right)/2\right)\floor{m_{i}/2}$
vertices in each $W_{\left(i,h\right)}$, leaving a total of $L=\sum_{i}\left(2L_{i}^{Z}+2L_{i}^{W}\right)$
vertices. (The reason we divide $m_{i}$ by 2 will become clear in
the proof of \ref{claim:find-templates}, to follow). Recall from
\ref{sub:fix-endpoints} that each $\rho^{-1}\le\left(\left(k-1\right)/2\right)\left(\left|Z_{\left[\bi\right]}\right|/\left|W_{\bi}\right|\right)\le\rho$,
so $m_{i}\ge\rho^{-1}\left|Z_{i}\right|$ and 
\[
L_{i}^{Z}/\left|Z_{i}\right|\ge\left(1+o\left(1\right)\right)\rho^{-1}/2,\quad L_{\left[\bi\right]}^{W}/\left|W_{\bi}\right|\ge\left(1+o\left(1\right)\right)\rho^{-2}/2.
\]
Provided we can find suitable templates, we can therefore apply \ref{lem:adjusting-paths}
with $\gamma=\rho^{-2}/3$. We can actually choose our templates almost
arbitrarily; we wrap up the details in the following claim.
\begin{claim}
\label{claim:find-templates}For large $k$, we can choose good templates
in such a way that if we remove the vertices of special paths corresponding
to these templates, there will be $L_{i}^{Z}$ special pairs left
in each $Z_{i}$ and $L_{\left[\bi\right]}^{W}$ vertices left in
each $W_{\bi}$.\end{claim}
\begin{proof}
It is convenient to describe our templates by a collection of $M:=\left(n-L\right)/\left(k+1\right)$
disjoint good special sequences. The equivalence classes of those
special sequences will give our templates.

Arbitrarily choose subsets $Z_{i}'\subseteq Z_{i}$ and $W_{\bi}'\subseteq W_{\bi}$
with $\left|Z_{i}'\right|=\left|Z_{i}\right|-L_{i}^{Z}$ and $\left|W_{\bi}'\right|=\left|W_{\bi}\right|-L_{\left[\bi\right]}^{W}$.
Let $X_{\bi}'\subseteq X_{\bi}$ be the subset of special vertices
in $X_{\bi}$ coming from special pairs in $Z_{\left[\bi\right]}'$.
Let $Z'=\bigcup_{i}Z_{i}'$ be the set of all special pairs in our
subsets, let $X'=\bigcup_{\bi}X_{\bi}'$ be the set of all special
vertices in those special pairs, and let $W'=\bigcup_{\bi}W_{\bi}'$
be the set of free vertices in our subsets. Note that $\left|Z'\right|=M=\left|X'\right|/2=\left|W'\right|/\left(k-1\right)$.

Recall that $\rho^{-1}\le\left(\left(k-1\right)/2\right)\left(\left|X_{\bi}\right|/\left|W_{\bi}\right|\right)\le\rho$
and note that $\left|W_{\bi}'\right|\ge\left|W_{\bi}\right|-m_{i}/2\ge\left|W_{\bi}\right|/2$.
So, if $k$ is large then for each $\bi$ we have
\[
\frac{\left|W_{\bi}'\right|}{\left|X_{\bi}'\right|}\ge\frac{\left|W_{\bi}\right|/2}{\left|X_{\bi}\right|}\ge\frac{k-1}{4\rho}\ge1.
\]
That is, for each $\bi$ there are at least as many vertices in $W_{\bi}'$
as in $X_{\bi}'$. (The only reason we divided $m_{i}$ by 2 in the
definitions of $L_{i}^{Z}$ and $L_{i}^{W}$ was to guarantee this).
We now start filling our special sequences with the vertices in $X'\cup W'$
(note $\left|X'\cup W'\right|=\left(k+1\right)M$, as required). For
each special sequence we choose its first and last vertex to be the
vertices of a special pair in $Z'$ (say we choose $\left\{ x,y\right\} \in Z_{i}'$,
and we choose $x\in X_{\left(i,1\right)}$ and $y\in X_{\left(i,2\right)}$
to be the first and last vertices in the special sequence respectively).
Then we choose arbitrary vertices from $W_{\left(i,1\right)}'$ and
$W_{\left(i,2\right)}'$ to be the second and second-last vertices
in the sequence (in our case, a vertex from $W_{\left(i,2\right)}'$
should be second and a vertex from $W_{\left(i,1\right)}'$ should
be second-last). We have just confirmed that there are enough vertices
in the $W_{\bi}$ for this to be possible. After we have determined
the first, second, second last and last vertices of every sequence,
we can use the remaining $\left(k-3\right)M$ vertices in $W'$ to
complete our special sequences arbitrarily.\end{proof}
\begin{rem}
If we are careful, instead of choosing our templates arbitrarily we
can strategically choose them in such a way that each cluster is only
involved in a few different templates. In the proof of \ref{lem:adjusting-paths},
we would then not need to divide each cluster into nearly as many
as $Q^{k+1}$ subsets. This would mean that we can avoid the use of
the powerful results of \cite{GKRS07} and make do with a simpler
and weaker variant of \ref{lem:random-preserves-superregular}.
\end{rem}
We now apply \ref{lem:adjusting-paths} using the templates from \ref{claim:find-templates},
and remove the vertices of the special paths that we find. Update
the $W_{\bi}$, $X_{\bi}$ and $Z_{i}$; we finally have $2\left|W_{\bi}\right|=\left(k-1\right)\floor{m_{\left[\bi\right]}/2}=\left(k-1\right)\left|X_{\bi}\right|$
for each $\bi$, and each $\left(X_{\bi},W_{\flip{\bi}}\right)$ and
$\left(W_{\bi},W_{\flip{\bi}}\right)$ are $\left(\varepsilon_{4},\delta_{4}\right)$-super-regular,
where $\delta_{4}=\delta_{3}/4$ and $\varepsilon_{4}=3\rho^{2}\varepsilon_{3}$.

\subsubsection{Completing the embedding with the blow-up lemma\label{sub:blow-up}}

We have now finished with the random edges in $\R$; what remains
is sufficiently well-structured that we can embed the remaining paths
using the blow-up lemma and the super-regularity in $\G$. We first
give a statement of the blow-up lemma, due to Koml\'os, S\'ark\"ozy
and Szemer\'edi.
\begin{lem}[Blow-up Lemma \cite{KSS97}]
Let $\delta,\Delta>0$ and $r\in\NN$. There is $\varepsilon=\varepsilon\left(r,\delta,\Delta\right)$
such that the following holds.

Let $C$ be a graph on the vertex set $\range r$, and let $n_{1},\dots,n_{r}\in\NN^{+}$.
Let $V_{1},\dots,V_{r}$ be pairwise disjoint sets of sizes $n_{1},\dots,n_{r}$.
We construct two graphs on the vertex set $V=\bigcup_{i=1}^{r}V_{i}$
as follows. The first graph $b_{n_{1},\dots,n_{r}}\left(C\right)$
(the ``complete blow-up'') is obtained by putting the complete bipartite
graph between $V_{i}$ and $V_{j}$ whenever $\left\{ i,j\right\} $
is an edge in $C$. The second graph $B\subseteq b_{n_{1},\dots,n_{r}}\left(C\right)$
(a ``super-regular blow-up'') is obtained by putting edges between
each such $V_{i}$ and $V_{j}$ so that $\left(V_{i},V_{j}\right)$
is an $\left(\varepsilon,\delta\right)$-super-regular pair.

If a graph $H$ with maximum degree bounded by $\Delta$ can be embedded
into $b\left(C\right)$, then it can be embedded into $B$.
\end{lem}
Now we describe the setup to apply the blow-up lemma. For each $i\in\range q$,
define 
\[
S_{i}^{0}=X_{\left(i,1\right)},\;S_{i}^{1}=W_{\left(i,2\right)},\;S_{i}^{2}=W_{\left(i,1\right)},\;S_{i}^{3}=X_{\left(i,2\right)}.
\]
We construct an auxiliary graph $G_{i}$ as follows. Start with the
subgraph of $\G$ induced by $S_{i}^{0}\cup S_{i}^{1}\cup S_{i}^{2}\cup S_{i}^{3}$,
and remove every edge not between some $S_{i}^{\ell}$ and $S_{i}^{\ell+1}$.
Identify the two vertices in each special pair to get a super-regular
``cluster-cycle'' with 3 clusters $S_{i}^{Z}$, $S_{i}^{1}$ and
$S_{i}^{2}$. To be precise, the elements of $S_{i}^{Z}$ are the
$\left|Z_{i}\right|$ (ordered!) special pairs $\left(\x_{1},\x_{2}\right)\in X_{\left(i,1\right)}\times X_{\left(i,2\right)}$;
such a special pair is adjacent to $v\in S_{i}^{\h}$ if $\x_{\h}$
is adjacent to $v$. The purpose of defining $G_{i}$ is that a cycle
containing exactly one vertex $\left(\x,\y\right)$ from $S_{i}^{Z}$
corresponds to a special path connecting $\x$ and $\y$. See \ref{fig:cluster-paths}.

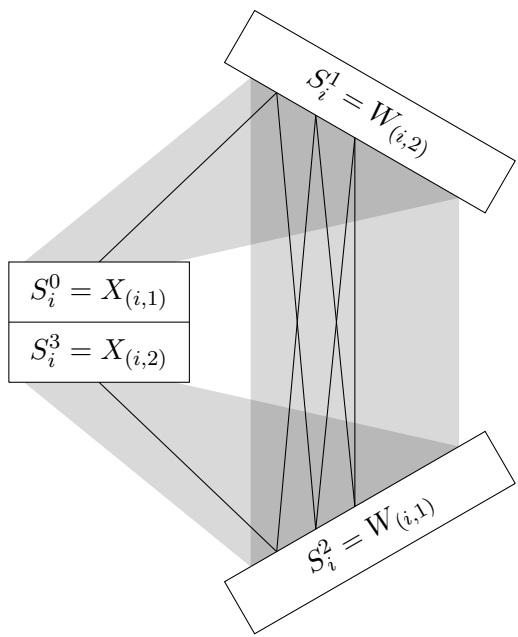
\begin{figure}[h]
\begin{center}
\begin{tikzpicture}[scale=0.8]

\def\sqrt{1.41421356237}

\def\rh{1}
\def\rw{3}

\draw (0,-\rh) -| (\rw,\rh) -| (0,-\rh);
\draw (0,0) -- (\rw,0);

\node[shape=coordinate] (xb11) at (\rw-0.25,-\rh) {};
\node[shape=coordinate] (xb12) at (0.25,-\rh) {};
\node[shape=coordinate] (xb21) at (0.25,\rh) {};
\node[shape=coordinate] (xb22) at (\rw-0.25,\rh) {};
\node[shape=coordinate] (xc1) at (\rw/2,-\rh) {};
\node[shape=coordinate] (xc2) at (\rw/2,\rh) {};

\node at (\rw/2,\rh/2) {$S_i^0=X_{(i,1)}$};
\node at (\rw/2,-\rh/2) {$S_i^3=X_{(i,2)}$};

\def\rw{5}
\def\rsx{6}
\def\rsy{3.5}
\FPeval{\rshifty2}{\rsy*2}

\def\rx{6}
\def\rr{30}
\FPeval{\rrot2}{\rr+180}

\begin{scope}[shift={(\rsx,-\rsy)},rotate=\rrot2]

\def\rn{1}
\node[shape=coordinate] (r\rn1) at (-\rw/2,0.5){};
\node[shape=coordinate] (r\rn2) at (-\rw/2,-0.5){};
\node[shape=coordinate] (r\rn3) at (\rw/2,-0.5){};
\node[shape=coordinate] (r\rn4) at (\rw/2,0.5){};
\draw (r\rn1) -- (r\rn2) -- (r\rn3) -- (r\rn4) -- (r\rn1);
\node[shape=coordinate] (r{\rn}p1) at (-\rw*0.3,-0.5){};
\node[shape=coordinate] (r{\rn}p2) at (-\rw*0.15,-0.5){};
\node[shape=coordinate] (r{\rn}p3) at (0,-0.5){};
\node[shape=coordinate] (r{\rn}p4) at (\rw*0.15,-0.5){};
\node[shape=coordinate] (r{\rn}p5) at (\rw*0.3,-0.5){};
\node[shape=coordinate] (r{\rn}b1) at (-\rw/2+0.5,-0.5){};
\node[shape=coordinate] (r{\rn}b2) at (\rw/2-0.5,-0.5){};

\fill [opacity=0.15] (xb\rn1) -- (r{\rn}b1) -- (r{\rn}b2) -- (xb\rn2) -- cycle;

\begin{scope}[rotate=-\rrot2]
\begin{scope}[shift={(0,\rshifty2)},rotate=-\rr]

\def\rn{2}
\node[shape=coordinate] (r\rn1) at (-\rw/2,0.5){};
\node[shape=coordinate] (r\rn2) at (-\rw/2,-0.5){};
\node[shape=coordinate] (r\rn3) at (\rw/2,-0.5){};
\node[shape=coordinate] (r\rn4) at (\rw/2,0.5){};
\draw (r\rn1) -- (r\rn2) -- (r\rn3) -- (r\rn4) -- (r\rn1);
\node[shape=coordinate] (r{\rn}p1) at (-\rw*0.3,-0.5){};
\node[shape=coordinate] (r{\rn}p2) at (-\rw*0.15,-0.5){};
\node[shape=coordinate] (r{\rn}p3) at (0,-0.5){};
\node[shape=coordinate] (r{\rn}p4) at (\rw*0.15,-0.5){};
\node[shape=coordinate] (r{\rn}p5) at (\rw*0.3,-0.5){};
\node[shape=coordinate] (r{\rn}b1) at (-\rw/2+0.5,-0.5){};
\node[shape=coordinate] (r{\rn}b2) at (\rw/2-0.5,-0.5){};

\fill [opacity=0.15] (xb\rn1) -- (r{\rn}b1) -- (r{\rn}b2) -- (xb\rn2) -- cycle;

\fill [opacity=0.15] (r{1}b2) -- (r{2}b1) -- (r{2}b2) -- (r{1}b1) -- cycle;
\draw (xc1) -- (r{1}p5);
\draw (r{1}p5)-- (r{2}p2);
\draw (r{2}p2) -- (r{1}p3);
\draw (r{1}p3) -- (r{2}p3);
\draw (r{2}p3) -- (r{1}p4);
\draw (r{1}p4) -- (r{2}p1);
\draw (r{2}p1) -- (xc2);

\end{scope}
\end{scope}

\end{scope}


\node[rotate=\rr] at (\rsx,-\rsy) {$S_i^2=W_{(i,1)}$};
\node[rotate=-\rr] at (\rsx,\rsy) {$S_i^1=W_{(i,2)}$};

\end{tikzpicture}
\end{center}

\caption{\label{fig:cluster-paths}An example special path is shown, for $k=7$.}
\end{figure}

Let $n_{i}^{Z}=\left|S_{i}^{Z}\right|=\left|Z_{i}\right|$. Note that
$G_{i}$ is a $\left(\varepsilon_{4},\delta_{4}\right)$-super-regular
blow-up of a 3-cycle $C_{3}$, with part sizes 
\[
n_{i}^{Z},\quad n_{i}^{1}:=\left(\left(k-1\right)/2\right)n_{i}^{Z},\quad n_{i}^{2}:=\left(\left(k-1\right)/2\right)n_{i}^{Z}.
\]
The corresponding complete blow-up $b_{n_{i}^{Z},n_{i}^{1},n_{i}^{2}}\left(C_{3}\right)$
contains $n_{i}^{Z}$ vertex-disjoint $k$-cycles (each has a vertex
in $S_{i}^{Z}$ and its other $k-1$ vertices alternate between $S_{i}^{1}$
and $S_{i}^{2}$). By the blow-up lemma, $G_{i}$ also contains $n_{i}^{Z}$
vertex-disjoint $k$-cycles. Since $k$ is odd, each of these must
use at least one vertex from $S_{i}^{Z}$, but since $\left|S_{i}^{Z}\right|=n_{i}^{Z}$,
each cycle must then use exactly one vertex from $S_{i}^{Z}$. These
cycles correspond to special paths which complete our embedding of
$\T$.

\section{Concluding Remarks}

We have proved that any given bounded-degree spanning tree typically
appears when a linear number of random edges are added to an arbitrary
dense graph. There are a few interesting questions that remain open.
Most prominent is the question of embedding more general kinds of
spanning subgraphs into randomly perturbed graphs. It would be particularly
interesting if the general result of \cite{BST09} (concerning arbitrary
spanning graphs with bounded degree and low bandwidth) could be adapted
to this setting. There is also the question of universality: whether
it is true that a randomly perturbed dense graph typically contains
every bounded-degree spanning tree at once. Finally, it is possible
that our use of Szemer\'edi's regularity lemma could be avoided,
thus drastically improving the constants $\c\left(\a\right)$ and
perhaps allowing us to say something about random perturbations of
graphs which have slightly sublinear minimum degree.

\noindent \textbf{Acknowledgement.} Parts of this work were carried
out when the first author visited the Institute for Mathematical Research
(FIM) of ETH Zurich, and also when the third author visited the School
of Mathematical Sciences of Tel Aviv University, Israel. We would
like to thank both institutions for their hospitality and for creating
a stimulating research environment.

\providecommand{\bysame}{\leavevmode\hbox to3em{\hrulefill}\thinspace}
\providecommand{\MR}{\relax\ifhmode\unskip\space\fi MR }
\providecommand{\MRhref}[2]{%
  \href{http://www.ams.org/mathscinet-getitem?mr=#1}{#2}
}
\providecommand{\href}[2]{#2}

\end{document}